\documentclass[conference]{IEEEtran}
\usepackage{amsmath,amsfonts,amsthm,pict2e,graphicx,balance,xcolor,url,placeins}
\usepackage{algorithm}
\usepackage[endLComment=~]{algpseudocodex}

\usepackage[font=footnotesize]{caption}
\theoremstyle{definition}

\newcommand{\invGi}{ \ensuremath \mathrm{inv}\Gamma}

\DeclareMathOperator{\sign}{sign}
\providecommand{\upstrut}[1]{\rule[-1pt]{0pt}{#1}}

\newlength{\funcwd}\newlength{\funcht}

 \newcommand{\invmathsvs}[1]{
\mbox{
\settowidth{\funcwd}{$#1$}\settoheight{\funcht}{$#1$}%
\raisebox{.93\funcht}{%
 \makebox[0pt][l]{\hspace{.15\funcwd}{$\scriptscriptstyle \vee$} }}$#1$}
 }
 \newcommand{\invGs}%
 {\mbox{\settowidth{\funcwd}{$\Gamma$}\settoheight{\funcht}{$\Gamma$}%
\raisebox{.90\funcht}{%
\makebox[0pt][l]{\hspace{.15\funcwd}{$\scriptscriptstyle \vee$} }}%
\scalebox{1}[0.875]{$\Gamma$}}}

\newcommand{\vstrut}[2]{\rule[#1]{0in}{#2}}

\newtheorem{Th}{Theorem}

\begin{document}
\title{The Inverse of the Complex Gamma Function}
\author{
\IEEEauthorblockN{D. J. Jeffrey}
\IEEEauthorblockA{Ontario Research Centre for Computer Algebra \\ and Department of Mathematics \\
University of Western Ontario, Canada\\
\texttt{djeffrey@uwo.ca}
}
\and
\IEEEauthorblockN{Stephen M. Watt}
\IEEEauthorblockA{Ontario Research Centre for Computer Algebra \\ and Cheriton School of Computer Science \\
University of Waterloo, Canada\\
\texttt{smwatt@uwaterloo.ca}
}
}

\maketitle
\begin{abstract}
We consider the functional inverse of the Gamma function in the complex plane, where it is multi-valued, and define a set of suitable branches by proposing a natural extension from the real case.
\end{abstract}
\section{Introduction}
Professor James Davenport, whose 70th birthday is being celebrated at this conference, has been a prolific investigator
of multivalued functions. His work has covered several topics in the area, dating back more than two decades---see, for example, \cite{Davenport1,Davenport2,Davenport3,Davenport4}.  The present article follows in this tradition.

A recent review of the Gamma function and its properties
comments that the inverse function has received little study~\cite{CorlessBorwein}.
Some basic properties
were given in \cite{Pedersen2015} and \cite{Amenyo}, but the structure of the branches in the complex plane has not been considered.  This article provides some results in this area.

The Gamma function for complex $z$ can be defined by \cite{DLMF} 
\begin{equation} \label{def:integral}
	\Gamma(z+1) =z! = \int_0^\infty t^z\mathrm{e}^{-t}\,\mathrm{d}t  \quad \mbox{for} \quad \Re z > -1 \ ,
\end{equation}	
together with Euler's reflection formula
\begin{equation}\label{eq:reflect}
  \Gamma(1-z) \Gamma(z) = \frac {\pi}{\sin \pi z}\ .
\end{equation}	
The $\Gamma$ notation shifts the argument with respect to factorial by $1$, as shown in \eqref{def:integral};
this shift, often called a \emph{minor but continual nuisance} \cite{CorlessBorwein},
is usually blamed on Legendre; but in fact Euler did this first \cite{gronau2003}.
We denote the inverse of $\Gamma$ by either
$w = \invGs_k(z)$, or $w=\invGi_k(z)$, where $k$ will label the branch, which will be defined here.

We remind the reader of some basic properties of $\Gamma$. Figure~\ref{fig:gamma} shows a plot of the $\Gamma$-function.
The derivative of $\Gamma$ introduces the digamma function, $\Psi$, through the logarithmic derivative:
\[ \frac{d}{dx} \ln \Gamma(x)=\Psi(x)\ ,\]
thus making $\Gamma^\prime\!=\Psi\Gamma$. The complex conjugate is $\overline{\Gamma(z)} = \Gamma(\overline z)$.
\begin{figure}
\centering
\includegraphics[scale=0.5]{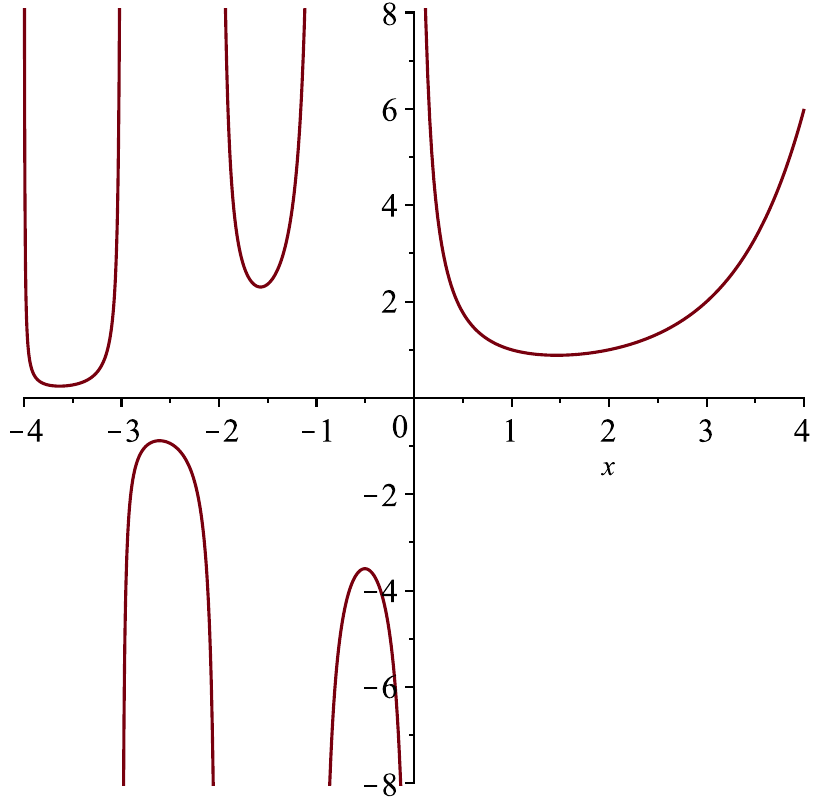}
\caption{A graph of the $\Gamma$ function. Note the poles at non-positive integers.}
\label{fig:gamma}
\end{figure}

In \cite{CorlessBorwein}, Stirling's approximation is used to give an asymptotic approximation to $\invGs_0$ using the Lambert $W$ function.
\begin{equation} \label{eq:StirApproxInv}
    \invGs(x)\approx \frac12 + \frac{\ln\left(x/\sqrt{2\pi}\right)}{W\left(e^{-1}\ln\left(x/\sqrt{2\pi}\right)\right)}\ .
\end{equation}
This approximation is remarkably accurate, even for relatively small arguments. For example, it gives
$\invGs(24)\approx 4.99$, where the exact answer is $5$. Trying a smaller value, we see $\invGs(4)\approx 3.653$.
This can be checked by comparing $\Gamma(3.653)=3.948$ with $4$.
Finally,
we note that Pedersen~\cite{Pedersen2015} has shown that one branch of the inverse can be extended to the complex plane,
and it can be represented as a Pick function.

\section{Definition of real branches}
We first consider the inverse of $\Gamma$ as a function of the reals.
The graph of $\invGs$ is show in Figure \ref{fig:inverse}, as produced
using the Maple parametric plot command:
\begin{small}
\begin{verbatim}
  plot([GAMMA(x),x,x=-4..3.5], discont=true);
\end{verbatim}
\end{small}
The turning points of $\Gamma$ define the boundaries of the branches, and the partition of the range of $\invGs$.
To define the branches, we introduce some notation.
Since the Gamma function has extremal points at
\begin{equation}\label{eq:derivzero}
  \frac{d\Gamma(z)}{dz}=\Gamma(z)\Psi(z) = 0\ ,
\end{equation}
we denote the points $(\psi_i,\gamma_i)$ by the definitions
\begin{align}\label{def:turning}
  \Psi(\psi_0)=0\ , & \mbox{ where }\psi_0>0\ , \\
  \Psi(\psi_k)=0\ , & \mbox{ where } k<0 \mbox{ and } k<\psi_k<k+1\ ,
\end{align}
together with $\gamma_k=\Gamma(\psi_k)$.
Numerical values are displayed in Table~\ref{tab:branchpoints}.  Further values are given in Table~5.4.1 of~\cite{DLMF}, with our $k$ corresponding to their $-n$.
\begin{figure}
  \centering
  \includegraphics[width=.875\columnwidth]{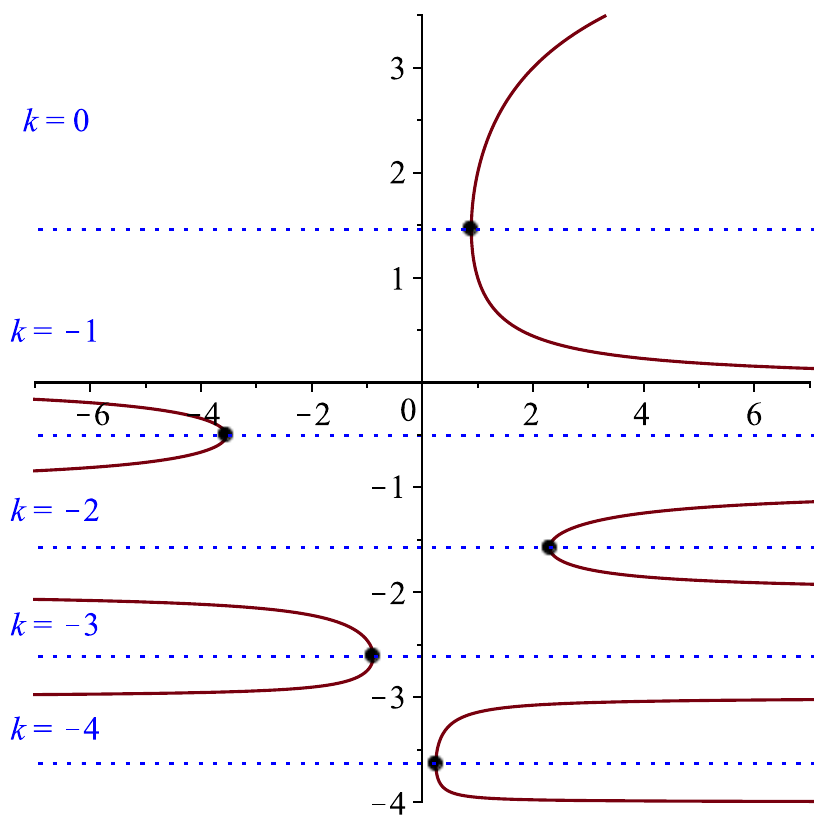}
  \caption{Inverse of the $\Gamma$ function showing the critical points and the division of the range into branches.}
\label{fig:inverse}
\end{figure}

The branches are defined by their ranges as displayed in Table~\ref{tab:branchdefs}.
For the general approach to branched functions, see \cite{Jeffrey2014AISC}.
The numbering of the branches is chosen so that $k=0$ is the principal branch, while the remaining branches receive negative
numbers so that the labels follow roughly the ranges taken by the branches.
Algorithm~\ref{algo:real_branch} determines the domain of $x$ values corresponding to a value of $\Gamma(x)$ and branch number.

Some special cases are worth noting. 
Since $\Gamma(2)=1!=\Gamma(1)=0!=1$, the corresponding inverses
must lie on different branches.
This requires $\invGs_0(1)=2$ and $\invGs_{-1}(1)=1$.
Some other special values of interest are $\invGs_{0}(\sqrt{\pi}/2)=3/2$, which lies very close
to the branch point, and
$\invGs_{-1}(\sqrt{\pi})=1/2$.

\section{Extension to the Complex Plane}
We follow \cite{Jeffrey2017} and~\cite{JeffreyWatt2022} in considering extensions to the complex plane. We start with the domain of $\Gamma(z)$, which will become the range of $\invGs$.
We wish to partition $\mathbb C$ so that $\Gamma$ is injective
on each element of the partition.  
We argue that the parallel lines $z = \psi_k + iy$ partition $\mathbb C$ in this way.
\begin{Th}
Let $\mathcal D_0 = \{z \;|\; \psi_0 \le \Re z\}$ and
$\mathcal D_k = \{z \;|\; \psi_k \le \Re z < \psi_{k+1}\}$
for $k < 0$.  Then $\Gamma: \mathcal D_k \rightarrow \mathbb C$ is injective for $k \le 0$.
\end{Th}
\begin{proof}
We start with the observation that all zeros of $\Psi$ are real, and that $\Psi$ is analytic everywhere except on the non-positive integers.
We show we cannot have two points $z_1, z_2 \in \mathcal D_k$ with $z_1 \ne z_2$ but $\Gamma(z_1) = \Gamma(z_2)$.   
Since $\Gamma$ has no zeros, it is useful to work with the entire function $1/\Gamma$  and we require $1/\Gamma(z_1) = 1/\Gamma(z_2)$.  
Consider $1/\Gamma$ restricted to 
$\{z \; | \; z = L(\lambda) = (1-\lambda) z_1+ \lambda z_2, \lambda \in [0,1]\}$,
that is a line in $\mathcal D_k$ from $z_1$ to $z_2$.
Then the real and imaginary parts of $1/\Gamma(L(\cdot))$ are both $C^\infty[0,1]$
real-valued functions so $N(\lambda) = |1/\Gamma(L(\lambda))|^2$ 
is a differentiable real function on $[0,1]$ with $N(0) = N(1)$.  
Bearing in mind that $z_1 \ne z_2$, by Rolle's theorem we must have 
$N^\prime(\lambda) = 0$ for some $\lambda \in (0,1)$.   
But $N^\prime(\lambda) = 0$ only when $\Psi(L(\lambda)) = 0$,
\textit{i.e.} for $z = \psi_k$, which is the unique point in $\mathcal D_k$ for which $\Psi(z) = 0$.  So unless the line segment $L: [0,1] \rightarrow \mathcal D_k$ crosses $\psi_k$, $\Gamma(z_1) \ne \Gamma(z_2)$.  
\begin{table}[t]
  \centering
  \begin{tabular}{|c|c|c|}
  \hline
    $k$ & $\psi_k$ & $\gamma_k$ \upstrut{2.6ex}\\[1pt]
     \hline
   \upstrut{9pt}$\phantom- 0$ & $\phantom- 1.461632$ & $\phantom- 0.885603$ \\
   \upstrut{9pt}$ -1$         & $         -0.504083$ & $         -3.544644$ \\
   \upstrut{9pt}$ -2$         & $         -1.573498$ & $\phantom- 2.302407$ \\
   \upstrut{9pt}$ -3$         & $         -2.610720$ & $         -0.888136$ \\
   \upstrut{9pt}$ -4 $        & $         -3.635293$ & $\phantom- 0.245127$ \\
   \upstrut{9pt}$ -5 $        & $         -4.653237$ & $         -0.052780$ \\
   \hline
  \end{tabular}\vspace{1.5ex}
  \caption{Critical values for branches of \vstrut{0pt}{11pt}$\invmathsvs{\Gamma}_k$. 
  These points are plotted as $(\gamma_k,\psi_k)$
  in Figure \ref{fig:inverse}.}
  \label{tab:branchpoints}
\end{table}
\begin{table}[t]
  \centering
  \begin{tabular}{|c|c|c|}
     \hline
     $k$ & Range condition & Argument range \upstrut{2.6ex}\\[1pt]
     \hline
    $\phantom- 0$
      & \upstrut{12pt}
        \(\psi_0 \le \invGs_0 \phantom{< \psi_0} \)
      & \(g \ge \gamma_{0\phantom{-}}\) \\
    $-1$
      & \upstrut{23pt}
        \(\begin{aligned}
           0         &<   \invGs_{-1} < \psi_0 \\ 
           \psi_{-1} &\le \invGs_{-1} < 0
        \end{aligned}\) 
      & \(\begin{aligned}
           g &> \gamma_0 \\
           g &\le \gamma_{-1}
        \end{aligned}\) \\
    $-2$
      & \upstrut{23pt}
        \(\begin{aligned}
          -1       &<   \invGs_{-2} < \psi_{-1} \\
          \psi_{-2}&\le \invGs_{-2} <-1
        \end{aligned}\)
      & \(\begin{aligned}
           g &< \gamma_{-1} \\
           g &\ge \gamma_{-2} 
        \end{aligned}\)
    \\[10pt]\hline
    \end{tabular}\vspace{1.5ex}
  \caption{Branches of inverse real $\Gamma$ with notation $g=\Gamma(x)$ and
  $x=\invmathsvs{\Gamma}(g)$. If $g$ falls outside the intervals shown, there is no real value for inverse Gamma.}\label{tab:branchdefs}
\end{table}
\begin{algorithm}[t]
\caption{$\Gamma$ Real Domain Selection and Inverse by Index\upstrut{2.2ex}}
\label{algo:real_branch}
\begingroup
\newcommand\hi{\ensuremath{\text{hi}}}
\newcommand\lo{\ensuremath{\text{lo}}}
\newcommand\errorkw{\textbf{error }}
\newcommand\xorkw{\textbf{xor }}
\newcommand\orkw{\textbf{or }}
\newcommand\andkw{\textbf{and }}
\begin{algorithmic}
\LComment{Determine the $x$ domain given $g = \Gamma(x)$ and index $k$.}
\LComment{Returns bounds as pair $(lo, hi)$, $lo \in \mathbb R$, $hi \in \mathbb R \cup +\infty$.}
\Function{{\sc RealGammaDomain}}{$g \in \mathbb R, k \in \mathbb Z_{\le 0}$\upstrut{2.5ex}}
\LComment{Exclude nonexistent branches.\upstrut{2ex}}
\If{$k \ge 1 \text{ \orkw} g = 0 \text{ \orkw} (k = 0 \text{ \andkw} g < 0)$ } \errorkw \EndIf

\LComment{How much to adjust end points near $k = 0$.\upstrut{3.5ex}}
\If{$k = 0$} $\lo_0 \gets 1;\; \hi_0 \gets +\infty$
\ElsIf{$k = -1 \text{ \andkw} g > 0$}  $\lo_0 \gets 0;\; \hi_0 \gets 1$
\Else\ $\lo_0 \gets 0;\; \hi_0 \gets 0$
\EndIf

\LComment{Determine which side of the pole.\upstrut{3.5ex}}
\If{$k$ is even \xorkw $g > 0$}
       \State $\lo \gets k+1+\lo_0$
       \State $\hi \gets \text{solve } \Psi(x) = 0 \text{ for } x \in (\lo, k+2+\hi_0)$
       \If{$|g| < |\Gamma(\hi)|$}  \errorkw \EndIf
\Else
       \State $\hi \gets k + 1 + \hi_0$
       \State $\lo \gets \text{solve } \Psi(x) = 0 \text{ for } x \in (k+\lo_0, \hi)$
       \If{$|g| < |\Gamma(\lo)|$} \errorkw \EndIf
\EndIf
\State \Return (\lo, \hi)
\EndFunction

\LComment{Compute real inverse $\Gamma$ on given branch.\upstrut{4ex}}
\Function{{\sc RealInvGamma}}{$g \in \mathbb R, k\in \mathbb Z_{\le 0}$}\upstrut{2.5ex}
    \State $(\lo, \hi) \gets \text{{\sc RealGammaDomain}}(g, k)$
    \State \Return solve $\Gamma(x) = g$ for $x\in (\lo, \hi)$
\EndFunction
\end{algorithmic}
\endgroup
\end{algorithm}

The segment $L$ will cross $\psi_k$ only if $\Re z_1 = \Re z_2 = \psi_k$.
It remains to show that we cannot have $\Gamma(z_1) = \Gamma(z_2)$ in this case.
By the same argument as earlier, we cannot have $\Im z_1$ and $\Im z_2$ with the same sign, since that would require a zero of $\Psi$ off the real line.  Therefore $\Im z_1$ and $\Im z_2$ must have opposite signs.
Without loss of generality, let $z_1 = \psi_k + iy_1$ and $z_2 = \psi_k - iy_2$ with $y_1, y_2 \in \mathbb R_{>0}$. 
We observe that $\Im \Gamma(\phi_k + i\eta)$
must have a fixed sign for all $\eta > 0$ and a fixed sign for all $\eta < 0$.   
Otherwise, by continuity, $\Gamma$ would have a real value off the real line, which does not occur. 
Since \[\overline{\Gamma(z_2)} = \Gamma(\overline{z_2}) = \Gamma(\psi_k + iy_2),\]
we must have 
\[\sign \Im \Gamma(z_1) = \sign \Im \overline{\Gamma(z_2)} = -\sign \Im \Gamma{z_2} \]
but 
\(\Gamma(z_1) = \Gamma(z_2)\) implies
\(  -\sign \Im \Gamma(z_2) = -\sign \Im \Gamma(z_1) \)
so $y_1 = y_2 = 0$ and $z_1 = z_2$, contradicting our hypothesis.
\end{proof}

\begin{figure}
    \centering
    \includegraphics[width=.8\columnwidth]{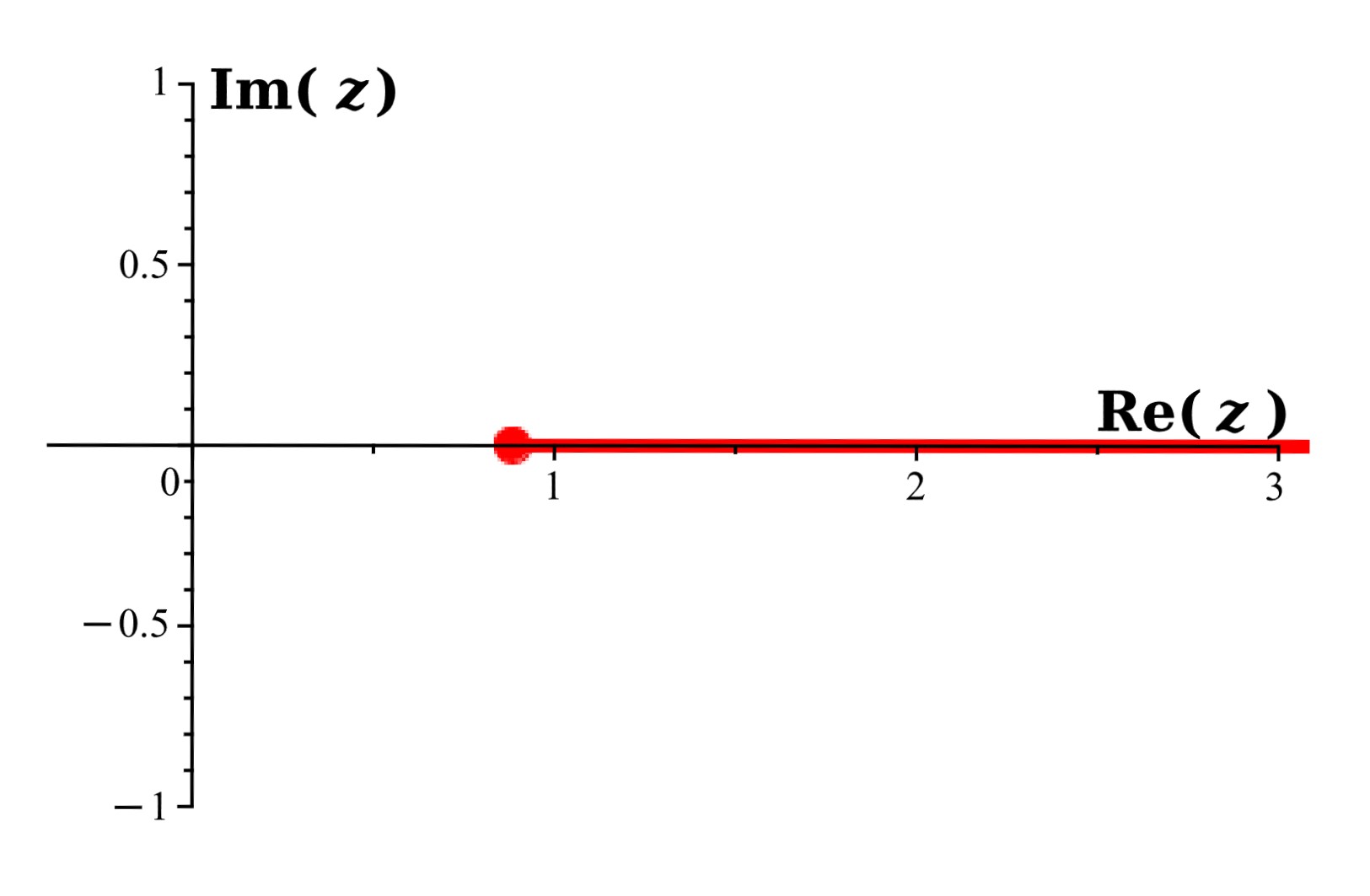}
    \\[1.5\baselineskip]
    \includegraphics[width=.775\columnwidth]{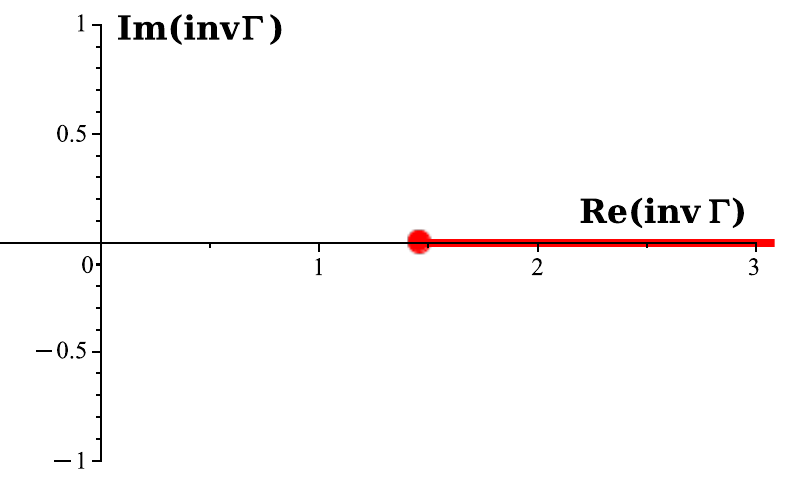}
    \caption{The domains and ranges for $\Gamma$ and $\invGs$. On the real axes are marked the values taken by the principal branch $\invGs_0$.}
    \label{fig:branches1}
\end{figure}
\begin{figure}
    \centering
    \includegraphics[width=.85\columnwidth]{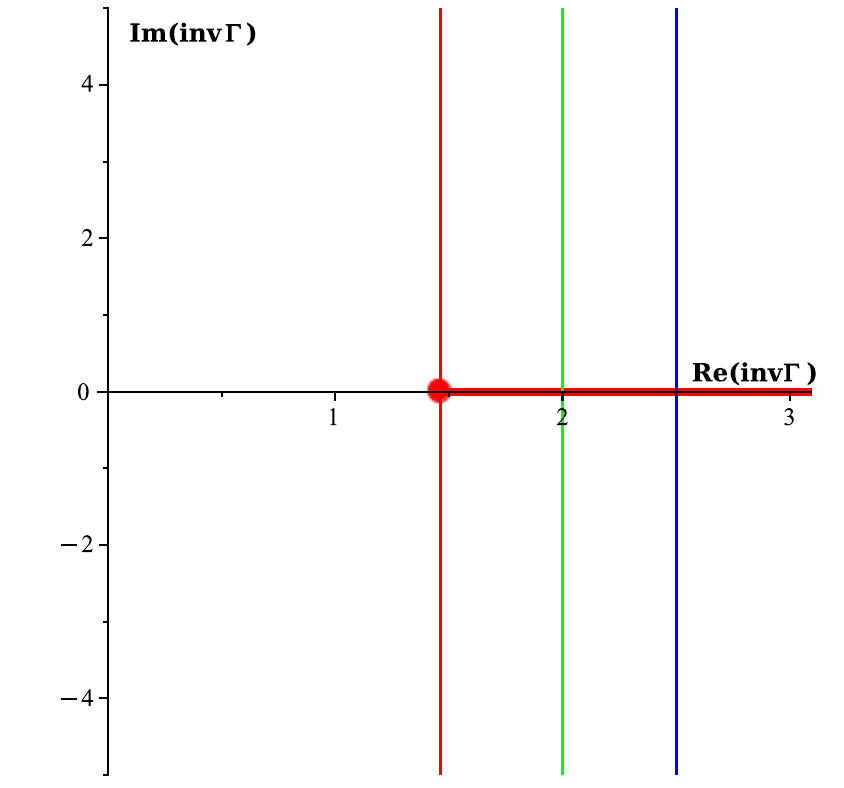}
    \caption{Contours in $\mathcal D_0$ to visualize the range of the $\invGs_0$ principal branch.}
    \vspace{.8\baselineskip}
    \label{fig:branches2a}
\end{figure}
\begin{figure}
    \centering
    \hspace{5mm}\includegraphics[width=.875\columnwidth]{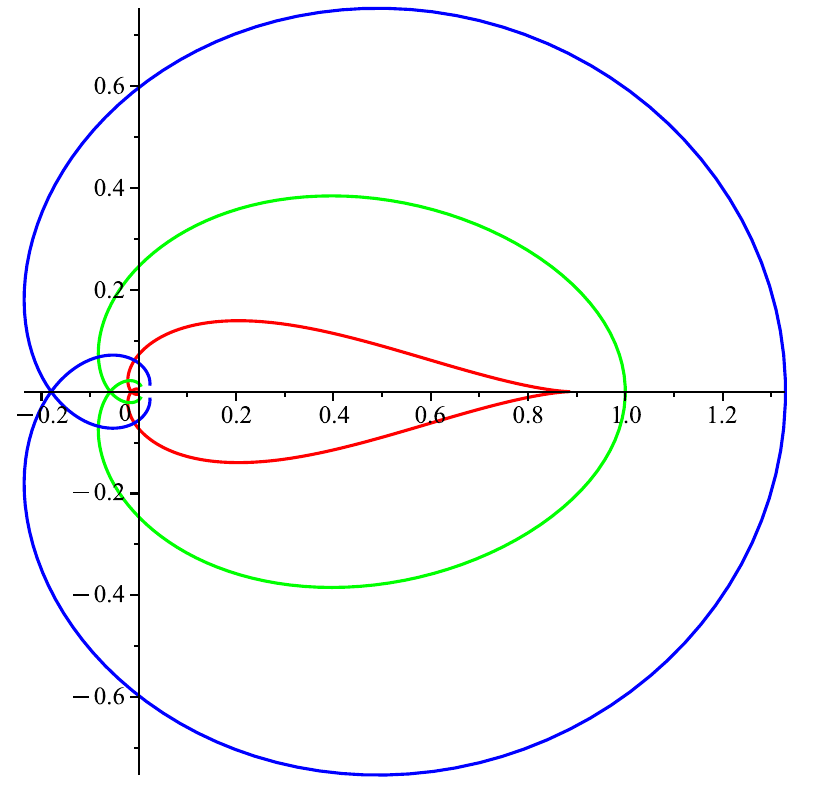}
    \vspace{.4\baselineskip}
    \caption{Contours of Figure~\ref{fig:branches2a} after mapping with $\Gamma$. Where a contour intersects itself is the edge of the branch of $\invGs$ in its range. 
    By magnifying the region around the origin, one can see that the red contour intersects 
    itself and the axis to the left of the origin. }
    \label{fig:branches2b}
\end{figure}
\begin{figure}
    \centering
    \hspace{5mm}\includegraphics[width=.975\columnwidth]{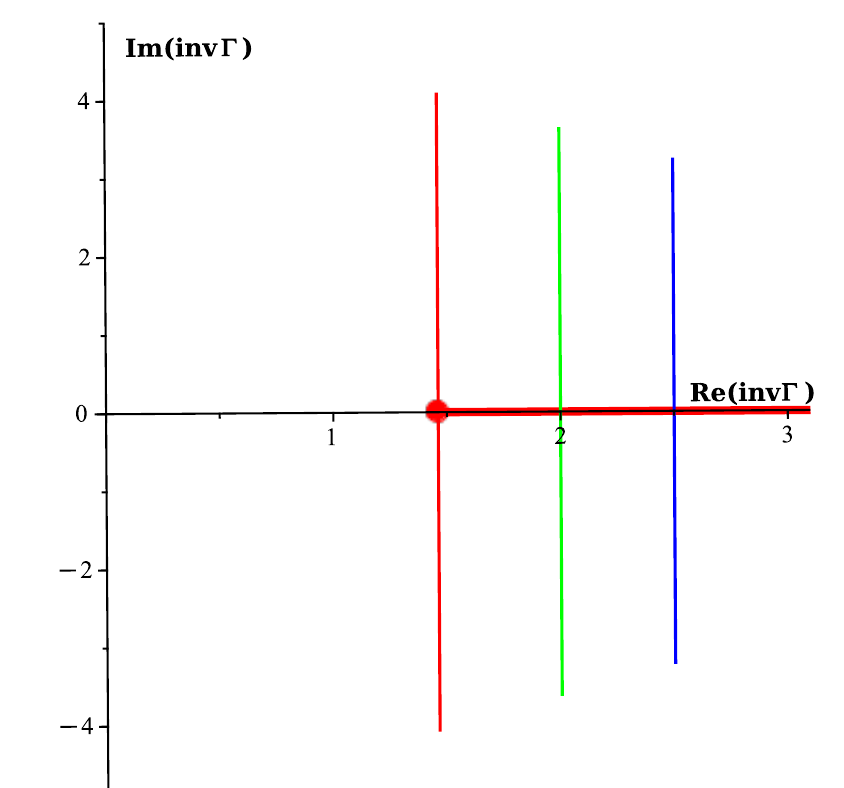}
    \caption{Corrected contours in $\mathcal D_0$ that visualize the range of the $\invGs_0$ principal branch.}
    \label{fig:G02trim}
\end{figure}
\begin{figure}
    \centering    \includegraphics[width=.875\columnwidth]{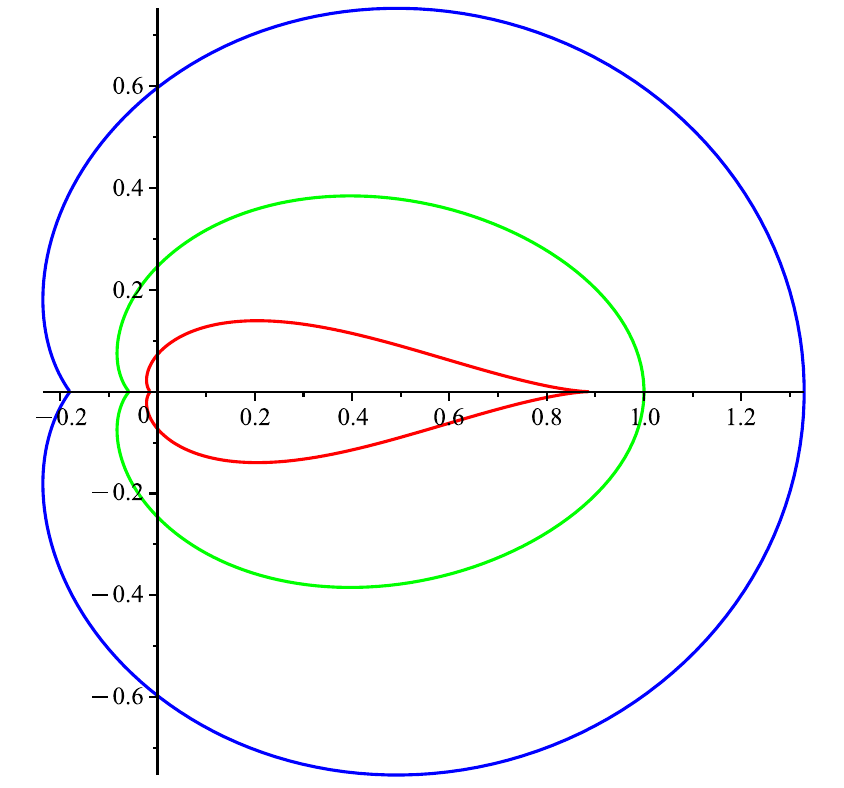}
    \vspace{2mm}
    \caption{Corrected contours of Figure~\ref{fig:G02trim} after mapping with $\Gamma$. The contours no longer intersect and stop at the axis.}
    \label{fig:Go2trim}
\end{figure}
\begin{figure}
    \vspace{1.7mm}
    \centering
    \includegraphics[width=0.875\columnwidth]{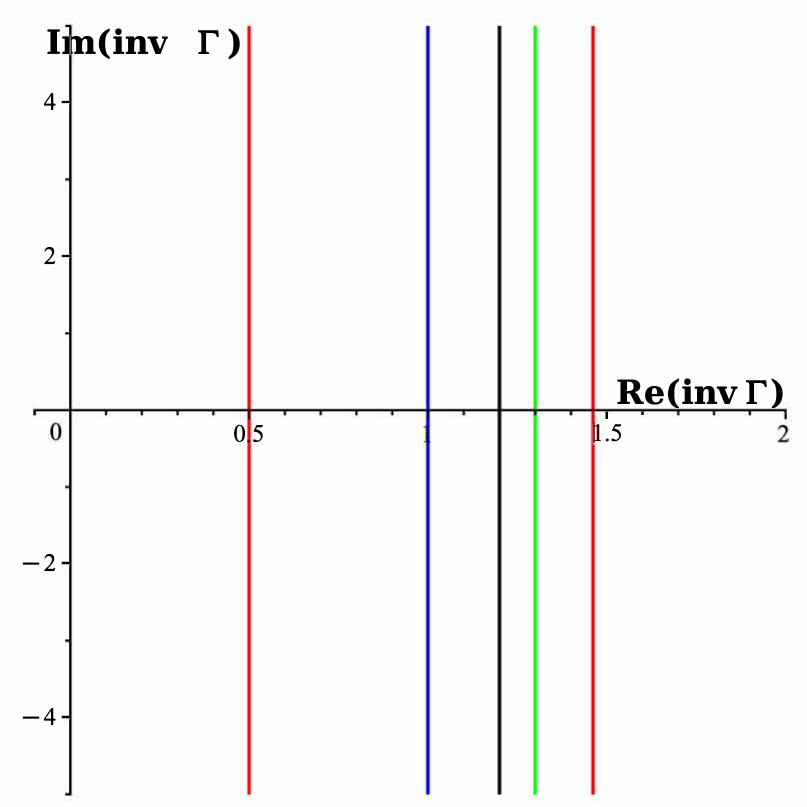}
    \vspace{1.3mm}
    \caption{Exploring the domain of the principal branch by contours which include forbidden parts of the
    domain of $\invGs$ (principal branch).}
    \label{fig:G02a}
\end{figure}

\begin{figure}
    \centering
    \includegraphics[width=0.855\columnwidth]{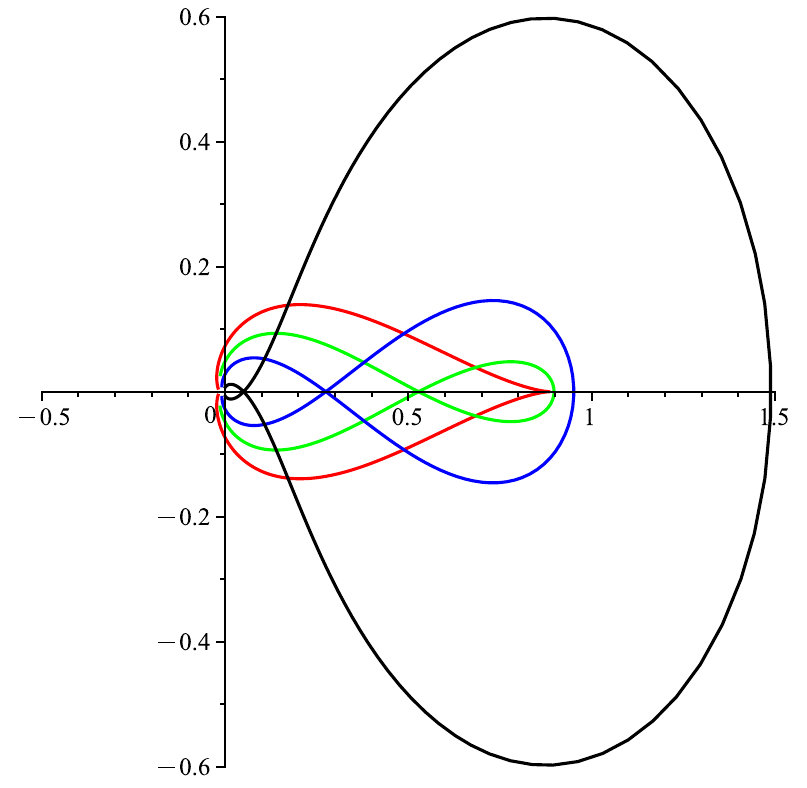}
    \caption{Contours from figure \ref{fig:G02a} showing intersections from sections of the contour
    touching the real axis. The contours in figure \ref{fig:G02a} must be trimmed both near the axis and at the upper end.}
    \label{fig:invG02a}
\end{figure}

\section{Graphical Representation}
In order to obtain a graphical representation of the branches of $\Gamma$, we begin by
setting up two complex planes: the range of $\Gamma$, which is also the domain of $\invGs$, and vice versa. On these we mark the intervals applying to the real values of the principal branch.
In Figure \ref{fig:branches1}, the top axes are the domain of $\invGs$ (and the range of $\Gamma$) and since $\invGs$ is real for $x\ge \gamma_0\approx 0.886$, this interval 
is marked in red on the real axis. The bottom axes are the range of $\invGs$ and the real axis is marked with the interval $\invGs\ge \psi_0\approx 1.46$, in which the real values of $\invGs$ lie,  
both sets of markings being in accord with Table~\ref{tab:branchdefs}.
We now draw contours in the lower axes and map them using $\Gamma$ to the domain in the top axes.

To move from real values to complex values, we erect contours on the plotted values of $\invGs$, as shown in Figure \ref{fig:branches2a}. The length of the contours is arbitrary at this stage (they are straight for convenience). We now apply $\Gamma$ to the contours and obtain the curves in Figure \ref{fig:branches2b} in the domain of $\invGs$. We notice that the curves self-intersect. Where they intersect is a discontinuity in function values; in other words a branch cut. The curves after the intersection are violating our requirement that the function is injective. Therefore we return to the straight contours in figure \ref{fig:branches2a} and trim their lengths until their images meet on the real axis without intersecting.
This stage of the construction is shown in figures \ref{fig:G02trim} and \ref{fig:Go2trim}.

We have now established that additional contours stretching to the right of those shown in
figure \ref{fig:G02trim} will account for all of the complex plane
 outside the contours shown in figure \ref{fig:Go2trim}. 
 This leaves the region inside the innermost (red) curve unaccounted for. In order to place contours inside the red curve, we have to add contours to Figure \ref{fig:G02trim} and them map them as before; 
there is, however, a difficulty. The values taken by the principal branch obey $\invGs \ge \gamma_0$, and those values
have already been used to build contours. 
Looking at figure \ref{fig:branches2a}, we see we cannot use the segment of the real axis 
that is not marked in red. We can, though, add contours there for exploration. These are shown in figure \ref{fig:G02a} and their mapping in figure \ref{fig:invG02a}.
Now the mapped contours self-intersect in two places, indicating that the contours need trimming at two places. The trimmed contours are shown in figure \ref{fig:G4c}, and the resulting inner region of the map is shown in figure \ref{fig:G4}.

Combining the results above, we obtain the complete description of the range, and hence the definition, of the principal branch; see Figure \ref{fig:branch3}. The construction we have been compiling has now revealed two things. In the domain of $\invGs_0$ we have two lines of discontinuity, where we had to avoid self-intersections. 
In other words, there are two branch cuts: $(-\infty,0)$ and
$(0,\gamma_0)$, as shown in Figure \ref{fig:cuts0}. 
\FloatBarrier

\begin{figure}
    \centering
    \includegraphics[width=0.875\columnwidth]{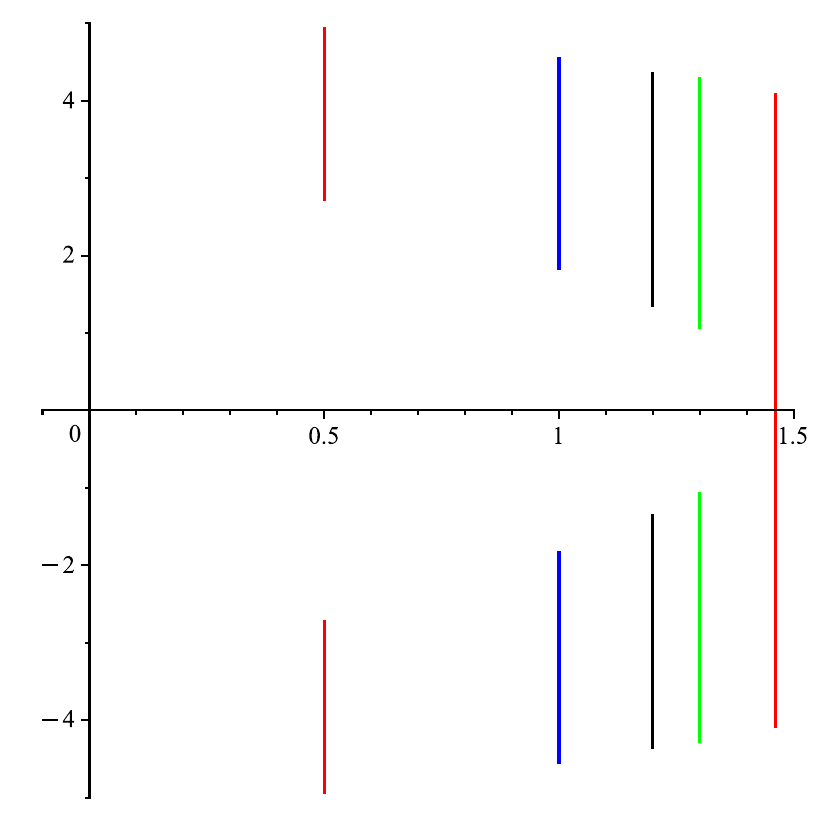}
    \caption{The contours of figure \ref{fig:G02a} trimmed to avoid self-intersections.}
    \label{fig:G4c}
\end{figure}
\FloatBarrier

\begin{figure}
    \centering
    \includegraphics[width=0.875\columnwidth]{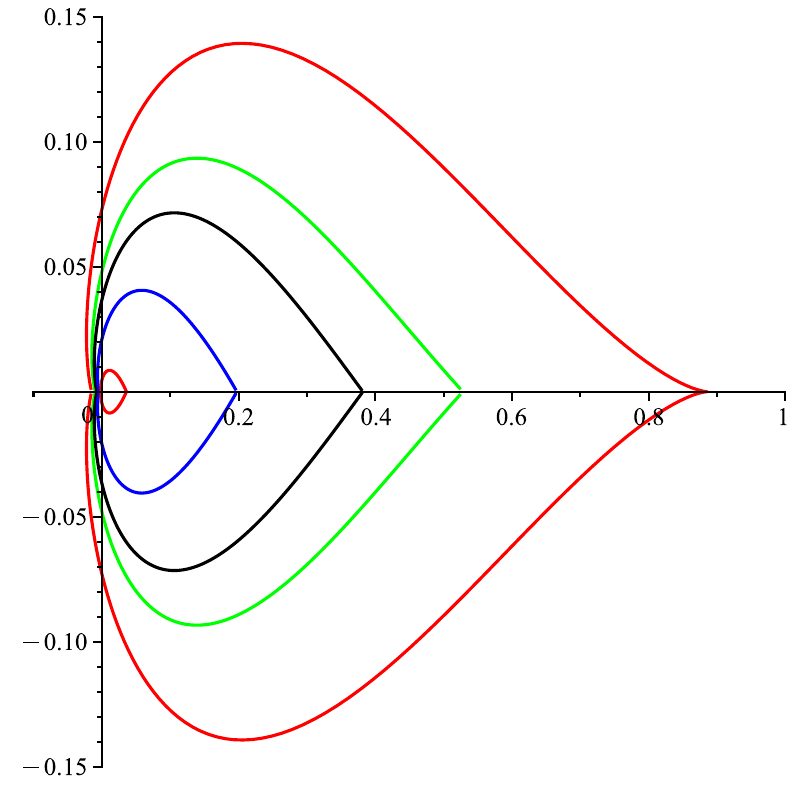}
    \caption{The contours in 
    figure \ref{fig:G4c}, mapped using $\Gamma$.
    Note that the left sides of the contours intersect the negative real axis.}
    \label{fig:G4}
\end{figure}

\begin{figure}
\centering    \includegraphics[width=0.75\columnwidth]{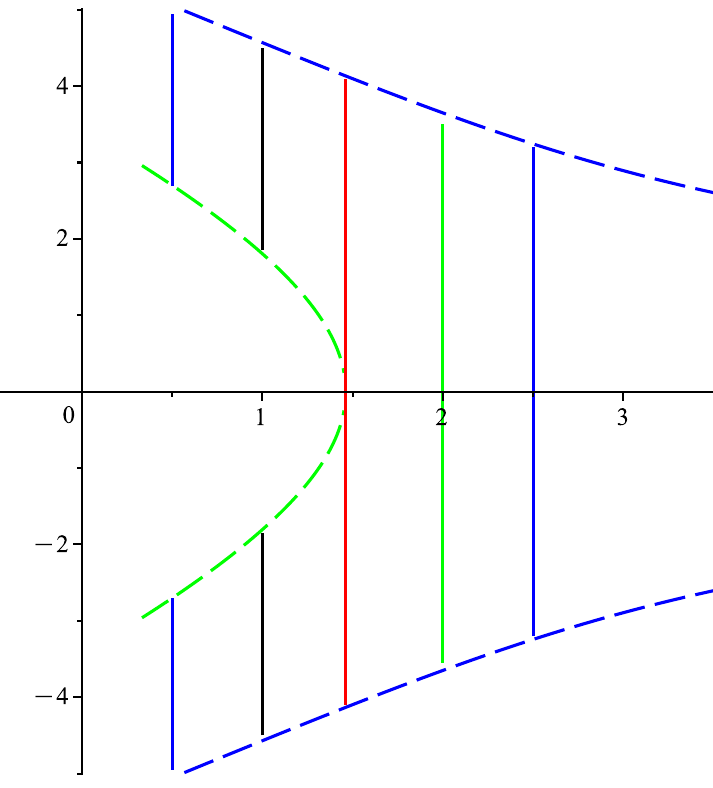}
\vspace{4pt}
\caption{The range in $\mathcal D_0$ of $\invGs_0$ as revealed by the contours constructed extending from the real data. The boundary curves (dashed) are obtained from mapping the branch cuts shown in Figure \ref{fig:cuts0} using $\invGs$. Note that the colours of the boundaries correspond to the colours of the branch cuts.}
 \label{fig:branch3}
\end{figure}
\begin{figure}
    \vspace{-1mm}
    \centering
    \includegraphics[width=0.875\columnwidth]{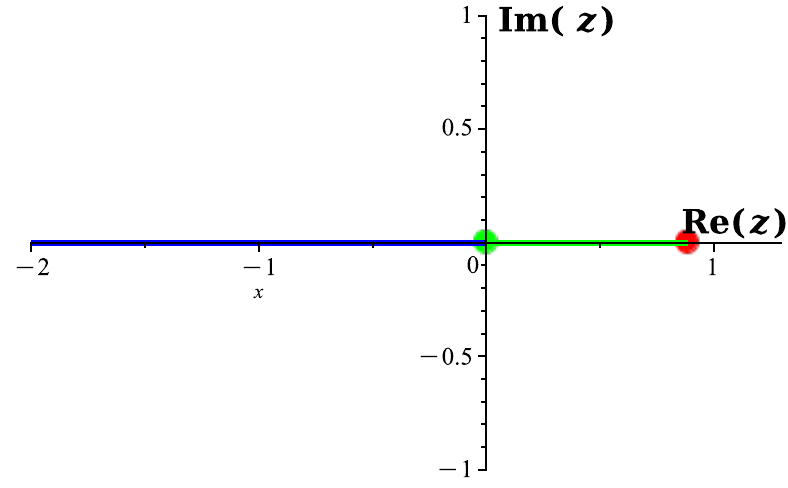}
    \vspace{12.9mm}
    \caption{The branch cuts for the principal branch $\invGs_0$. Note that although the two cuts
    appear to form a continuous line, they are separated by a singularity at the origin.}
    \vspace{-6mm}
    \label{fig:cuts0}
\end{figure}
\FloatBarrier

\begin{figure}
    \centering    \includegraphics[width=.83\columnwidth]{pix/G01.jpg}
    \\
\hspace{3.5mm}\includegraphics[width=.77\columnwidth]{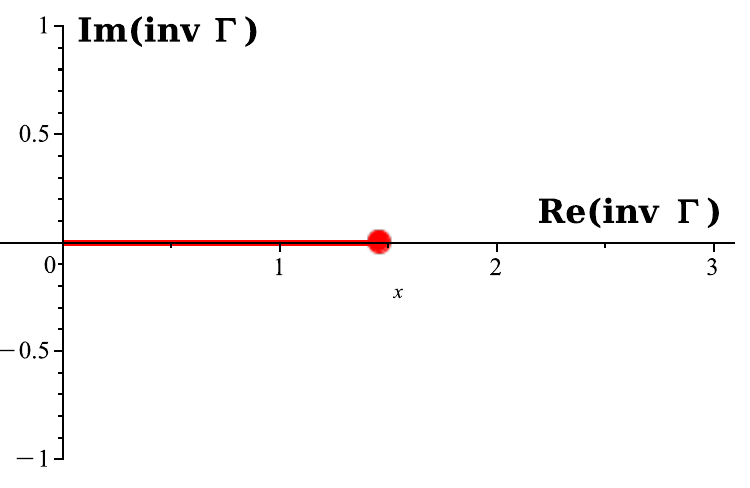}
    \caption{The domains and ranges for branch $k=-1$. In contrast to $\invGs_{0}$, now $\invGs_{-1}$ takes values between $0$ and $\psi_0$.}
    \label{fig:branches1m}
    
\end{figure}

\begin{figure}
    \centering
    \includegraphics[width=.85\columnwidth]{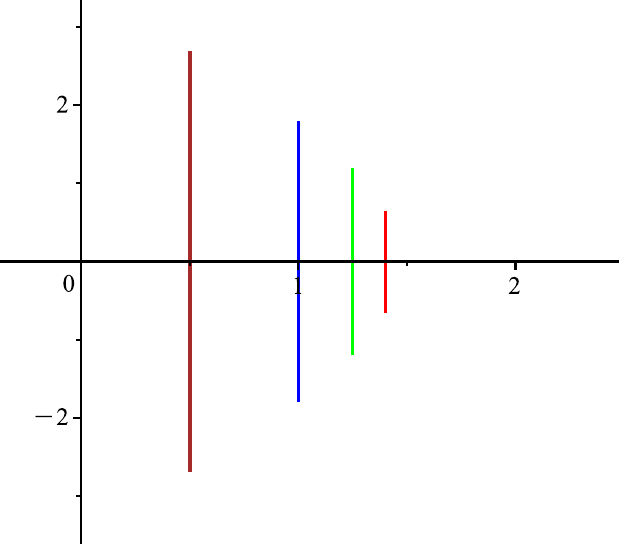}
    \caption{The range in $\mathcal D_{-1}$ of $\invGs_{-1}$ as revealed by the contours constructed extending from the real data. The range fits with Figure \ref{fig:branch3} and the two regions share a boundary.}
    \label{fig:branch3m}
\end{figure}

We can now 
map the branch cuts into the range of $\invGs$ using a crude Maple function for $\invGs$, and the result is a plot of the boundaries to the range of $\invGs_0$.
It should be noted that it is tempting to treat the two branch cuts as one, and simply say that there is a cut $-\infty<x\le \gamma_0$.

To show that the approach applies to all branches, we briefly consider branch $k=-1$.
This branch shares a singular point with branch $k=0$, namely $(\phi_0,\psi_0)$, but now values of $\invGs_{-1}$ decrease to zero. Thus we replace Figure \ref{fig:branches1}
with Figure \ref{fig:branches1m}. Similarly, Figure \ref{fig:branch3} is replaced by 
Figure \ref{fig:branch3m}, and we note that the two figures fit together like tectonic plates, and share a boundary.

\section{Conclusion}
In this paper, we have continued to develop the principle that a discussion of the branches of a multivalued function should start in the \textit{range} of the function. This is in contrast to the traditional treatments, for example, in \cite{DLMF},
\balance
where discussion starts by defining branch cuts in the \textit{domain} of the function (see, for example, \cite[Chap. 4]{DLMF} and their sections on logarithm and arctangent). For functions such as $\invGs$ or $W$, where the branch ranges are not regular, the branch cuts in the domain alone are not sufficient for understanding the structure of the branches. It is clear that the inverse $\Gamma$ function still possesses many interesting properties, which have not been touched yet.
These issues are being explored in depth at present.

\end{document}